\documentclass[11pt]{amsart}
\usepackage{amsfonts}
\usepackage{amssymb,latexsym}
\usepackage{enumerate}
\usepackage{mathrsfs}
\usepackage{amsmath}
\usepackage{amsthm}
\usepackage{hyperref}
\usepackage[square, comma, sort&compress, numbers]{natbib}
\usepackage{bm}
\usepackage{fancyhdr}
\usepackage{lastpage}
\usepackage{layout}
\usepackage{esint}
\usepackage{fancyhdr}
\newtheorem{thm}{Theorem}[section]
\newtheorem{prop}[thm]{Proposition}

\newtheorem{lemma}[thm]{Lemma}

\numberwithin{equation}{section}

\newcommand{\pt}{\partial}
\newcommand{\bp}{\overline{\partial}}

\newcommand{\pal}{\parallel}
\newcommand{\raw}{\rightarrow}

\newcommand{\mf}{\mathcal{F}}
\newcommand{\tr}{\mbox{tr}}

\newcommand{\lag}{\langle}
\newcommand{\rag}{\rangle}

\newcommand{\lmd}{\lambda}
\newcommand{\vps}{\varepsilon}
\newcommand{\vph}{\varphi}
\newcommand{\og}{\omega}

\newcommand{\nla}{\nabla}
\newcommand{\tg}{\triangle}
\newcommand{\wg}{\wedge}
\newcommand{\sqt}{\sqrt{-1}}

\newcommand{\lw}{\Lambda_{\og}}

\newcommand{\ptt}{\frac{\pt}{\pt t}}
\newcommand{\ptA}{\frac{\pt A}{\pt t}}

\newcommand{\mbr}{\mathbb{R}}

\newcommand{\mee}{\mathcal{E}}
\newcommand{\maa}{\mathcal{A}}

\newcommand{\mcc}{\mathcal{C}}

\begin{document}
\title{\textsc{The limiting behaviour of Hermitian-Yang-Mills flow over compact non-K\"ahler manifolds }}
\author{Yanci Nie and Xi Zhang}
\address{Yanci Nie\\School of Mathematical Sciences\\
Xiamen University\\
Xiamen, 361005\\ } \email{nieyanci@xmu.edu.cn}
\address{Xi Zhang\\Key Laboratory of Wu Wen-Tsun Mathematics\\ Chinese Academy of Sciences\\School of Mathematical Sciences\\
University of Science and Technology of China\\
Hefei, 230026,P.R. China\\ } \email{mathzx@ustc.edu.cn}
\thanks{The authors were supported in part by NSF in China,  No.11625106, 11571332, 11526212.}
\maketitle

\begin{abstract}
  In this paper, we analyze the asymptotic behaviour of the Hermitian-Yang-Mills flow over a compact non-K\"ahler manifold $(X,g)$ with the Hermitian metric $g$ satisfying the Gauduchon and Astheno-K\"ahler condition.
\end{abstract}
\section{Introduction}

Let $X$ be an $n$-dimensional compact complex manifold and $g$ a Hermitian metric with associated $(1,1)$-form $\og$. $g$ is called to be Gauduchon if $\og$ satisfies $\pt\bp \og^{n-1}=0$.  It has been proved by Gauduchon (\cite{Gaud}) that if $X$ is compact, there exists a Gauduchon metric in the conformal class of every Hermitian metric $g$. If $\pt\bp\og^{n-2}=0$, the Hermitian metric $g$  is said to be Astheno-K\"ahler  which was introduced by Jost and Yau in \cite{JY93}.

Let $(X,\og)$ be an $n$-dimensional compact complex manifold with $\pt\bp\og^{n-1}=0$ and $(L,h)$  a Hermitian line bundle over $X$.  The $\og$-degree of $L$ is defined by
\begin{equation*}
\deg_{\og}(L):=\int_X c_1(L, A_h)\wedge\displaystyle{\frac{\og^{n-1}}{(n-1)!}},
\end{equation*}
where $c_1(L,A_h)$ is the first Chern form of $L$ associated with the induced Chern connection  $A_h$.  Since $\pt\bp \og^{n-1}=0,$  $\deg_{\og}(L)$ is  well defined and independent of the choice of  metric $h$ (\cite[p.~34-35]{MA}). Now given a rank $s$ coherent analytic sheaf  $\mathcal{F}$, we consider the  determinant line bundle  $\det{\mathcal{F}}=(\wedge^s\mathcal{F} )^{**}$ associated with $\mathcal{F}.$  Define the $\og$-degree of $\mathcal{F}$ by
\begin{equation*}
\deg_{\og}(\mathcal{F}):=\mbox{deg}_{\og}(\det{\mathcal{F}}).
\end{equation*}
If $\mf$ is non-trivial and torsion free, the $\og$-slope of $\mathcal{F}$ is defined by
  $$\mu_{\og}(\mathcal{F})=\frac{\mbox{deg}_{\og}(\mathcal{F})}{\mbox{rank}(\mathcal{F})}.$$

Let $(E,\bp_E)$ be a rank $r$ holomorphic vector bundle over $(X,\og)$. A Hermitian metric $H$ on $E$ is said to be $\og$-Hermitian-Einstein if the Chern curvature $F_H$  satisfies the Einstein condition
 $$ \sqrt{-1}\Lambda_{\og}F_H=\lmd\cdot \mathrm{Id}_{E},$$
 where $\lmd=\displaystyle{\frac{2\pi \mu_{\og}(E)}{Vol(X)}}.$  When the $(1,1)$-form $\og$ is understood, we omit the subscript  $\og$ in the above definitions.

In this paper, we consider the following Hermitian-Yang-Mills flow on the holomorphic bundle $(E,\bp_E)$ with initial data $H(0)=H_0$,
\begin{equation}\label{eq0}
  H^{-1}\frac{\pt H}{\pt t}=-2\left(\sqt\lw F_{H}-\lmd \mbox{Id}_E\right)
\end{equation}
where $\lmd=\displaystyle{\frac{2\pi \mu_{\og}(E)}{Vol(X)}}$ and $F_H$ is the curvature of the Chern connection with respect to $H$. The Hermitian-Yang-Mills flow (\ref{eq0}) was introduced and studied by Donaldson in \cite{Don}. When $(X, \omega )$ is K\"ahler, Donaldson proved  the long time existence and uniqueness of the solution for (\ref{eq0}). Using this flow, Donaldson(\cite{Don87}) obtained the existence of the irreducible Hermitian-Einstein metrics on stable bundles over algebraic manifolds which was extended by Uhlenbeck and Yau (\cite{UhYau}) to the K\"ahler case. On general Hermitian manifolds, the second author (\cite{zhang05}) got the long-time existence and uniqueness of the solution of (\ref{eq0}).

Let's consider the Hermitian vector bundle $(E, H_{0})$. Denote  the space of connections of $E$
compatible with $H_{0}$ by
$\maa_{H_0}$ and the space of unitary integrable connections of $E$  by $\maa_{H_0}^{1,1}$. We denote by $\textbf{G}^{\mathbb{C}}$ (resp. $\textbf{G}$, where $\textbf{G}=\{\sigma \in \textbf{G}^{\mathbb{C}}| \sigma^{\ast H_0}\sigma=\mathrm{Id}\}$) the complex gauge group (resp. unitary gauge group) of the Hermitian
vector bundle $(E, H_{0})$. $\textbf{G}^{\mathbb{C}}$ acts on the space $\maa_{H_0}$ as follows: for
$\sigma \in \textbf{G}^{\mathbb{C}}$ and $A\in \maa_{H_0}$,
\begin{equation}\label{id2}
\overline{\partial }_{\sigma(A)}=\sigma \circ \overline{\partial
}_{A}\circ \sigma^{-1}, \quad \partial _{\sigma (A)}=(\sigma^{\ast
H_{0}})^{-1} \circ \partial _{A}\circ \sigma^{\ast H_{0}}.
\end{equation}

Following Donaldson's argument (\cite{Don}), we can show that the  Hermitian-Yang-Mills flow (\ref{eq0}) is gauge equivalent to the following heat flow
\begin{equation}\label{A.0}
\begin{cases}
\frac{\pt A(t)}{\pt t}=\sqrt{-1}(\bp_A-\pt_A)\Lambda_{\og}F_A,\\
A(0)=A_0,
\end{cases}
\end{equation}
 where $A_0=(\bp_E,H_0)$. In fact, there is a family of complex gauge transformations $\sigma (t)\in \textbf{G}^{\mathbb{C}}$ satisfying
$\sigma(t) ^{\ast H_{0}}\sigma (t)=h (t)=H_{0}^{-1}H(t)$, where $H(t)$ is the long time solution of the Hermitian-Yang-Mills flow (\ref{eq0}) with the initial metric $H_{0}$, such that $A(t)=\sigma (t)(A_{0})$ is the long time solution of the heat flow (\ref{A.0}) with the initial connection $A_{0}$.

 When the underground manifold $(X, \omega )$ is K\"ahler, it is easy to see that the heat flow (\ref{A.0}) is just the Yang-Mills flow by the K\"ahler identity. There are many interesting results on the convergence of the Yang-Mills flow, see references \cite{AtiyahBott, chsh1, chsh2, Daskalopous, DasWent, hong2004asymptotical, Sibley,LZZ3}.
 In this article, we study the limiting behaviour of the Hermitian-Yang-Mills flow (or the heat flow (\ref{A.0})) under the assumption that $\omega $  is Gauduchon and Astheno-K\"ahler. We first give some basic properties of the heat flow (\ref{A.0}) including energy inequality, monotonicity formula of certain quantities and small energy regularity. Then, following the argument of Hong and Tian in (\cite{hong2004asymptotical}) and using Bando and Siu's extension technique, we obtain the following convergence result of the  heat flow (\ref{A.0})

\begin{thm}\label{main}
  Let $(X,\og)$ be an $n$-dimensional compact Hermitian manifold with $\og$ satisfying $\pt\bp\og^{n-1}=\pt\bp\og^{n-2}=0$. Suppose $A(t)$ is the global smooth solution of the heat flow (\ref{A.0}) on the Hermitian vector bundle $(E, H_{0})$ with the initial data $A_0$ over $(X,\og)$. Then

\begin{enumerate}[(1)]
  \item For every sequence $t_k\raw \infty$, there is a subsequence $t_{k_j}$ such that as $t_{k_j}\raw \infty,$ $A(t_{k_j})$ converges modulo gauge transformations to a solution $A_{\infty}$ of  equation
\begin{equation}\label{eq17}
 D_A\lw F_A=0
\end{equation}
on Hermitian vector bundle $(E_{\infty}, H_{\infty})$ in $C^{\infty}_{loc}$ topology outside a subset $\Sigma\subset X$, where $\Sigma$ is a closed set of Hausdorff codimension at least $4$.

\item The limiting $(E_{\infty}, H_{\infty},\bp_{A_{\infty}})$ can be extended to the whole $X$ as a reflexive sheaf with a holomorphic orthogonal splitting
\begin{equation*}
  (E_{\infty}, H_{\infty}, A_{\infty})=\bigoplus_{i=1}^l(E_{\infty}^i,H_{\infty}^i,A_{\infty}^i),
\end{equation*}
where $H_{\infty}^i$ is an admissible Hermitian-Einstein metric on reflexive sheaf $E_{\infty}^i$.
\end{enumerate}

\end{thm}

The paper is organised as below. In Section 2, we  present some basic properties of the heat flow (\ref{A.0}). In Section 3, we give the detailed proof of Theorem \ref{main}.

\section{Existence of the heat flow and some basic estimates}
Let $(X,\og)$ be an $n$ dimensional compact Hermitian manifold and $(E, H_{0})$ a rank $r$ complex vector bundle over $(X,\og)$.   On the space $\maa_{H_0}$, we define the Yang-Mills functional by
\begin{equation*}
  \rm{YM}(A)=\int_X|F_A|^2\,dV_{\og}.
\end{equation*}
And the negative gradient flow of Yang-Mills functional is
\begin{equation*}
\ptA=-D_A^*F_A,
\end{equation*}
which is called the Yang-Mills flow.

 Using the Taylor expansion method, Demailly (\cite{demailly}) showed that for any $A\in \maa_{H_0}^{1,1}$, it holds
\begin{equation}\label{eq01}
\bp_A^*=-\sqrt{-1}[\Lambda_{\og}, \pt_A]-\overline{\tau}^*,\ \ \ \ \ \ \pt_A^*=\sqrt{-1}[\Lambda_{\og}, \bp_A]-\tau^*,
\end{equation}
where $\tau=[\Lambda_{\og}, \pt \og].$
From (\ref{eq01}), we know that the heat flow (\ref{A.0}) is equivalent to
\begin{equation}\label{A.3}
\begin{cases}
\frac{\pt A(t)}{\pt t}=-D_A^*F_A-[\Lambda_{\og}, dw]^*F_A,\\
A(0)=A_0.
\end{cases}
\end{equation}

Using the result in \cite{zhang05} and following the argument of Donaldson(\cite{Don}), we can obtain the long time  existence and uniqueness of solution of the heat flow (\ref{A.0}).
Since the proof is similar as that in \cite{Don}, we omit it.

\begin{thm}
Let $(X,\og)$ be an $n$-dimensional compact Hermitian manifold and $(E,H_0)$ a rank $r$ Hermitian vector bundle over $X$. Given any $A_0\in \mathcal{A}_{H_0}^{1,1}$, the heat flow (\ref{A.0}) has a unique long-time solution in the complex gauge orbit of $A_0$ with the initial data $A_0$.
\end{thm}

\subsection{Basic estimates}

Suppose that $A(t)$ is a smooth solution of the heat flow (\ref{A.0}) and $f$ a real smooth function over $X$. It holds that

\begin{equation}\label{A.2}
\begin{split}
\frac{d}{d t}\int_X f^2 |F_A|^2\, dV_g=&2 \mbox{Re}\int_X \left\langle f^2 F_A, D_A \frac{dA}{dt}\right\rangle\, dV_g\\
=&2\mbox{Re}\left\{\int_X f^2 \Big\lag D_A^*F_A, \frac{\pt A}{\pt t}\Big\rag\, dV_g-\int_X\Big\lag F_A, d f^2\wedge \frac{\pt A}{\pt t}\Big\rag\, dV_g\right\}\\
=&-2\int_X f^2\left|\frac{\pt A}{\pt t}\right|^2\, dV_g-2\mbox{Re}\int_X f^2\Big\lag (\tau+\bar{\tau})^*F_A, \frac{\pt A}{\pt t} \Big\rag\, dV_g\\
&-2\mbox{Re}\int_X\Big\lag F_A, d f^2\wedge \frac{\pt A}{\pt t}\Big\rag\, dV_g.
\end{split}
\end{equation}
Setting $f\equiv 1$ on $X$, we get
\begin{equation}
\begin{split}
\frac{d}{d t} \int_X  |F_A|^2 \, dV_g
=-2\int_X \left|\frac{\pt A}{\pt t}\right|^2 \, dV_g-2\mbox{Re}\int_X\Big\lag(\tau+\bar{\tau})^* F_A, \frac{\pt A}{\pt t}\Big\rag\, dV_g.
\end{split}
\end{equation}

\begin{prop}
  If the fundament form  $\og$ satisfies $\pt\bp\og^{n-1}=\pt\bp\og^{n-2}=0$, there holds
\begin{equation}
  \int_X\Big\lag(\tau+\bar{\tau})^* F_A, \frac{\pt A}{\pt t}\Big\rag\, dV_g=0.
\end{equation}
\end{prop}

\begin{proof}

From Proposition 4.1 in \cite{Mcnamara}, we have
\begin{equation}\label{tau}
\tau^*F_A=-\frac{*(\bp(\og^{n-2})\wedge F_A)}{(n-2)!}+\frac{*(\bp(\og^{n-1})\Lambda_{\og}F_A)}{(n-1)!}
\end{equation}
and
\begin{equation}\label{taubar}
\bar{\tau}^*F_A=-\frac{*(\pt(\og^{n-2})\wedge F_A)}{(n-2)!}+\frac{*(\pt(\og^{n-1})\Lambda_{\og}F_A)}{(n-1)!}.
\end{equation}

At first, by (\ref{A.0}),(\ref{tau}) and Stokes formula, we have
\begin{equation*}
\begin{split}
&\int_X\left\lag\tau^* F_A, \frac{\pt A}{\pt t}\right\rag\, dV_g\\
=&\int_X\left\lag \tau^* F_A,\sqrt{-1}\bp_A\Lambda_{\og}F_A \right\rag\, dV_g\\
=&\int_X\left\lag -\frac{*(\bp(\og^{n-2})\wedge F_A)}{(n-2)!}+\frac{*(\bp(\og^{n-1})\Lambda_{\og}F_A)}{(n-1)!}, \sqrt{-1}\bp_A\Lambda_{\og}F_A\right\rag\, dV_g\\
=&\int_X\sqrt{-1}\left\lag \frac{\bp_A^* *(\bp(\og^{n-2})\wedge F_A)}{(n-2)!}, \Lambda_{\og}F_A\right\rag\, dV_g-\int_X\sqrt{-1}\left\lag\frac{\bp_A^**(\bp(\og^{n-1})\Lambda_{\og}F_A)}{(n-1)!}, \Lambda_{\og}F_A \right\rag\, dV_g\\
=:&I+II.
\end{split}
\end{equation*}
Then, by simple calculation, we have
\begin{equation*}
  \begin{split}
  I&=\int_X\sqrt{-1}\Big\lag \frac{\bp_A^* *(\bp(\og^{n-2})\wedge F_A)}{(n-2)!}, \Lambda_{\og}F_A\Big\rag\\
&=\int_X\sqrt{-1}\Big\lag \frac{*\pt_A(\bp(\og^{n-2})\wedge F_A)}{(n-2)!}, \Lambda_{\og}F_A\Big\rag\\
&=\int_X\sqrt{-1}\Big\lag \frac{*(\pt\bp(\og^{n-2})\wedge F_A)}{(n-2)!}, \Lambda_{\og}F_A\Big\rag\\
&=0
\end{split}
\end{equation*}
and
\begin{equation*}
  \begin{split}
  II=&-\int_X\sqrt{-1}\left\lag\frac{\bp_A^**(\bp(\og^{n-1})\Lambda_{\og}F_A)}{(n-1)!}, \Lambda_{\og}F_A \right\rag\, dV_g\\
=&-\int_X \sqrt{-1}\left\lag\frac{*\pt_A(\bp(\og^{n-1})\Lambda_{\og}F_A)}{(n-1)!}, \Lambda_{\og}F_A \right\rag\, dV_g\\
=&-\int_X \sqrt{-1}\left\lag\frac{*(\pt\bp(\og^{n-1})\Lambda_{\og}F_A)}{(n-1)!}, \Lambda_{\og}F_A \right\rag\, dV_g
+\int_X\sqrt{-1}\left\lag \frac{*(\bp \og^{n-1}\wedge \pt_A \Lambda_{\og}F_A)}{(n-1)!}, \Lambda_{\og}F_A \right\rag\, dV_g\\
=&-\int_X \sqrt{-1}\mbox{tr}\left\{\Lambda_{\og}F_A \frac{(\bp \og^{n-1}\wedge \pt_A \Lambda_{\og}F_A)}{(n-1)!}\right\}\\
=&-\int_X\frac{\sqrt{-1}}{2}\pt\mid\Lambda_{\og}F_A\mid^2\wedge \frac{\bp \og^{n-1}}{(n-1)!}\\
=&0.
\end{split}
\end{equation*}

In the same way, we have
\begin{equation*}
\int_X\lag\bar{\tau}^* F_A, \frac{\pt A}{\pt t}\rag\, dV_g=0.
\end{equation*}

\end{proof}
Therefore, there holds that
\begin{lemma}\label{lmm1}
Let $A(t)$ be a solution of the heat flow (\ref{A.0}) with initial data $A_0$ over $X$. Then
\begin{equation}\label{eq6}
\mbox{\rm YM}(t)+2\int_0^t\int_X\left|\frac{\pt A}{\pt t}\right|^2=\mbox{\rm YM}(0).
\end{equation}
\end{lemma}
Let $f$ be a cut-off function with support inside $B_{2R}(x_0)$ and $f\equiv 1$ on $B_R(x_0)$ such that $0\leq f\leq 1$ and $\mid d f\mid\leq 2 R^{-1}$. Set $e(A)=\mid F_A\mid^2.$ From the identity (\ref{A.2}),  we have
\begin{equation}\label{eq6}
\begin{split}
\left| \frac{d}{dt}\int_X f^2 e(A)+2\int_X f^2 \left| \frac{\pt A}{\pt t}\right|^2\right|\leq C\left(\int_X f|df| |F_A|\left|\frac{\pt A}{\pt t}\right|+\int_X f^2|F_A|\left|\frac{\pt A}{\pt t}\right|\right)
\end{split}
\end{equation}
Then from (\ref{eq6}), we can deduce the following local energy estimates:
\begin{lemma}{\rm{(Local energy estimates)}}\label{lmm2}
  Suppose $A(t)$ is a smooth solution of the heat flow (\ref{A.3}). Fix $x_0\in X$ and $R\in \mbr^+$ such that $B_{2R}(x_0)\subset X$. Then for any two finite numbers $s$ and $\tau$, we have
 \begin{equation*}
    \begin{split}
      \int_{B_R(x_0)}e(A)(\cdot, s)\ dV_g\leq &\int_{B_{2R}(x_0)}e(A)(\cdot, \tau)\ dV_g+2\int_{\min\{s,\tau\}}^{\max\{s,\tau\}}\int_X\Big|\ptA\Big|^2\, dV_gdt\\
      &+C\left( \frac{|s-\tau|}{R^2}\rm{YM}(0)\int_{\min\{s,\tau\}}^{\max\{s,\tau\}}\int_X\Big|\ptA\Big|^2\ dV_gdt\right)^{\frac{1}{2}}\\
      &+\left( C|s-\tau|\rm{YM}(0)\int_{\min\{s,\tau\}}^{\max\{s,\tau\}}\int_X \Big|\ptA\Big|^2\right)^{\frac{1}{2}}.
    \end{split}
  \end{equation*}
\end{lemma}

\begin{proof}
 From (\ref{eq6}) and H\"older inequality, we have
\begin{equation*}
     \begin{split}
    &\Big|\int_{\min\{s,\tau\}}^{\max\{s,\tau\}} \left(\frac{d}{dt}\int_X f^2 e(A)+2\int_X f^2 \big| \frac{\pt A}{\pt t}\big|^2 \right)\Big|\\
     \leq &\int_{\min\{s,\tau\}}^{\max\{s,\tau\}}\Big| \frac{d}{dt}\int_X f^2 e(A)+2\int_X f^2 \big| \frac{\pt A}{\pt t}\big|^2 \Big|\\
     \leq &C/R\int_{\min\{s,\tau\}}^{\max\{s,\tau\}}\left(\int_X f^2 e(A)\ dV_g\right)^{\frac{1}{2}}\left(\int_X\Big|\ptA\Big|^2\right)^{\frac{1}{2}}\\
    &+C\int_{\min\{s,\tau\}}^{\max\{s,\tau\}}\left(\int_X f^2 e(A)\ dV_g\right)^{\frac{1}{2}}\left(\int_X\Big|\ptA\Big|^2 \right)^{\frac{1}{2}}\\
     \leq &\left( \frac{C|s-\tau|}{R^2}\rm{YM}(0)\int_{\min\{s,\tau\}}^{\max\{s,\tau\}}\int_X \Big|\ptA\Big|^2\right)^{\frac{1}{2}}\\
    &+  \left( C|s-\tau|\mbox{YM}(0)\int_{\min\{s,\tau\}}^{\max\{s,\tau\}}\int_X \Big|\ptA\Big|^2\right)^{\frac{1}{2}}.
\end{split}
   \end{equation*}
This indicates that
\begin{equation*}
    \begin{split}
       &\int_{B_{R}(x_0)}e(A)(\cdot, s)\ dV_g\\
       \leq& \int_{B_{2R}(x_0)} e(A)(\cdot, \tau)\ dV_g+2\int_{\min\{s,\tau\}}^{\max\{s,\tau\}}\int_X \Big|\ptA\Big|^2\ dV_gdt\\
        &+C\left( \frac{|s-\tau|}{R^2}YM(0)\int_{\min\{s,\tau\}}^{\max\{s,\tau\}}\int_X\Big|\ptA\Big|^2\ dV_gdt\right)^{\frac{1}{2}}\\
      &+\left( C|s-\tau|\mbox{YM}(0)\int_{\min\{s,\tau\}}^{\max\{s,\tau\}}\int_X \Big|\ptA\Big|^2\right)^{\frac{1}{2}}.
    \end{split}
  \end{equation*}
\end{proof}

\begin{lemma}
  Suppose $A(t)$ is a smooth solution of the heat flow (\ref{A.0}). Then it holds that
\begin{equation*}
  \Big(\ptt-2\sqt\lw \pt\bp\Big)|\lw F_A|^2=-2|D_A\lw F_A|^2\leq 0.
\end{equation*}
Furthermore, $\pal \lw F_A\pal_{L^2}^2(t)$ is decreasing along the flow and the $L^{\infty}$ norm of the $\lw F_A$ is bounded.
\end{lemma}

\begin{proof}

By simple calculation, we have

\begin{equation*}
\begin{split}
  \ptt |\lw F_A|^2
&=2\mbox{Re}\left\lag \ptt \lw F_A,\lw F_A\right\rag\\
&=2\mbox{Re}\sqt\lw\left\lag  (\pt_A\bp_A-\bp_A\pt_A)\lw F_A,\lw F_A\right\rag
\end{split}
\end{equation*}
and
\begin{equation*}
  \begin{split}
  \sqt\lw \pt\bp|\lw F_A|^2
&=\sqt\lw\pt\lag\bp_A \lw F_A,\lw F_A \rag+\sqt\lw\pt\lag \lw F_A,\pt_A\lw F_A\rag\\
&=\sqt\lw\lag \pt_A\bp_A\lw F_A,\lw F_A\rag+|\bp_A \lw F_A|^2+|\pt_A\lw F_A|^2\\
&\ +\sqt\lw\lag \lw F_A,\bp_A\pt_A\lw F_A\rag\\
&=\mbox{Re}\sqt\lw\lag(\pt_A\bp_A-\bp_A\pt_A)\lw F_A,\lw F_A\rag+|D_A\lw F_A|^2.
\end{split}
\end{equation*}
This implies that
\begin{equation}\label{eq7}
   \Big(\ptt-2\sqt\lw \pt\bp\Big)|\lw F_A|^2=-2|D_A\lw F_A|^2\leq 0.
\end{equation}

Using the maximum principle, we have
$$\sup\limits_X|\lw F_A|^2(\cdot,t)\leq \sup\limits_X|\lw F_A|^2(\cdot,0).$$

Integrating the two sides of (\ref{eq7}) over $X$, we have
\begin{equation*}
\begin{split}
  &\ptt\int_X|\lw F_A|^2\displaystyle{\frac{\og^{n}}{n!}}-2\int_X\sqt\lw\pt\bp |\lw F_A|^2\displaystyle{\frac{\og^{n}}{n!}}\\
=&\ptt\int_X|\lw F_A|^2\displaystyle{\frac{\og^{n}}{n!}}-2\int_X\sqt|\lw F_A|^2\displaystyle{\frac{\bp\pt\og^{n-1}}{(n-1)!}}\\
=&\ptt\int_X|\lw F_A|^2\displaystyle{\frac{\og^{n}}{n!}}=-\int_X|D_A\lw F_A|^2\displaystyle{\frac{\og^{n}}{n!}}\leq 0.
\end{split}
\end{equation*}
Therefore, $\pal \lw F_A\pal_{L^2}^2(t)$ is decreasing along the flow.
\end{proof}

\begin{lemma}
 Suppose $A(t)$ is a smooth solution of (\ref{A.0}) and set
\begin{equation*}
  I(t)=\int_X|D_A\lw F_A|^2.
\end{equation*}
Then it holds that $I(t)\raw 0$ as $t\raw \infty.$
\end{lemma}
\begin{proof}
From equation (\ref{A.3}), we have
\begin{equation*}
  \begin{split}
   \frac{d}{dt}D_A\lw F_A
&=\Big[\ptA,\lw F_A\Big]+D_A\lw \frac{\pt F_A}{\pt t}\\
&=\Big[\sqt(\bp_A-\pt_A)\lw F_A,\lw F_A\Big]+D_A\lw D_A\ptA\\
&=\Big[\sqt(\bp_A-\pt_A)\lw F_A,\lw F_A\Big]+D_A\lw D_A\sqt(\bp_A-\pt_A)\lw F_A.
\end{split}
\end{equation*}
So
\begin{equation*}
  \begin{split}
  \frac{d}{dt}I(t)
=&2\mbox{Re}\int_X\left\lag \frac{d}{dt}(D_A\lw F_A), D_A\lw F_A \right\rag\\
=&2\mbox{Re}\int_{X}\left\lag [\sqt(\bp_A-\pt_A)\lw F_A,\lw F_A],D_A\lw F_A \right\rag\\
&+2\mbox{Re}\int_X\left\lag D_A\lw D_A\sqt(\bp_A-\pt_A)\lw F_A,D_A\lw F_A \right\rag\\
=&2\mbox{Re}\int_{X}\left\lag [\sqt(\bp_A-\pt_A)\lw F_A,\lw F_A],D_A\lw F_A \right\rag\\
&-2\int_X|D_A^*D_A\lw F_A|^2-2\mbox{Re}\int_X\left\lag (\tau+\bar{\tau})^*D_A\lw F_A, D_A^*D_A\lw F_A \right\rag.
\end{split}
\end{equation*}
At first, one can easily check that
\begin{equation*}
  \mbox{Re}\int_{X}\lag [\sqt(\bp_A-\pt_A)\lw F_A,\lw F_A],D_A\lw F_A \rag\leq C(n,\mbox{rank}E, \pal \lw F_A\pal_{L^{\infty}})I(t).
\end{equation*}
Then, it holds
\begin{equation}\label{eqa1}
\begin{split}
  \int_X\lag D_A\lw F_A, D_A\lw F_A\rag&=\int_X\lag \lw F_A, D_A^*D_A\lw F_A \rag\\
&\leq \pal\lw F_A\pal_{L^{\infty}}\int_X |D_A^*D_A\lw F_A|\\
&\leq \pal\lw F_A\pal_{L^{\infty}}\mbox{Vol}(X)^{1/2}\left(\int_X  |D_A^*D_A\lw F_A|^2\right)^{1/2}.
\end{split}
\end{equation}
Inequality (\ref{eqa1}) implies
\begin{equation*}
  I(t)^2\leq \pal\lw F_A\pal^2_{L^{\infty}}\mbox{Vol}(X)\int_X  |D_A^*D_A\lw F_A|^2.
\end{equation*}

At last, it is easy to check that
\begin{equation*}
  \begin{split}
  \int_X\lag (\tau+\bar{\tau})^*D_A\lw F_A, D_A^*D_A\lw F_A \rag
&=\int_X\lag D_A\lw F_A,\lw d\og\wg D_A^*D_A\lw F_A \rag\\
&\leq  \int_X| D_A^*D_A\lw F_A|^2+C^2I(t).
\end{split}
\end{equation*}
From the above all, we have
\begin{equation*}
  \begin{split}
  \frac{d I(t)}{dt}
&\leq C I(t)-\int_X  |D_A^*D_A\lw F_A|^2+C^2 I(t)\\
&\leq C I(t)-CI(t)^2.
\end{split}
\end{equation*}
From equality (\ref{eq6}) and
\begin{equation*}
\int_X\Big|\ptA\Big|^2=\int_X|(\bp_A-\pt_A)\lw F_A|^2=\int_X|D_A\lw F_A|^2,
\end{equation*}
we have
\begin{equation*}
  \int_0^{\infty}I(t)<YM(0).
\end{equation*}
Using the technique in \cite[Prop.~6.2.14]{DonKro}, we have $I(t)\raw 0$ as $t\raw \infty.$

\end{proof}

\begin{lemma}\label{bochner}
  Suppose $A(t)$ is a global smooth solution of heat flow (\ref{A.0}). Then it holds that
\begin{equation*}
  (\tg-\ptt)|F_A|^2\geq 2|\nla_A F_A|^2-C(1+|F_A|+|Ric|+|Rm|)|F_A|^2-C|F_A||\nla_A F_A|,
\end{equation*}
where $C$ is a positive constant depending on the geometry of $X$.
\end{lemma}

\begin{proof}

First, using Bochner technique, we have
\begin{equation*}
  \tg|F_A|^2=-2\lag \nla_A^*\nla_A F_A,F_A \rag+2|\nla_A F_A|^2.
\end{equation*}
Then by simple calculation, we have
\begin{equation*}
    \begin{split}
      \ptt F_A
      &=D_A\ptA=-D_AD_A^*F_A-D_A\alpha,
    \end{split}
  \end{equation*}
where $\alpha=[\lw,d\og]^*F_A$. Combining with the following Weitzenb\"ock formula
\begin{equation*}
  \tg_A F_A=-D_AD_A^* F_A=\nla_A^*\nla_A F_A+Ric\sharp F_A+F_A\sharp F_A,
\end{equation*}
we have
\begin{equation*}
  \ptt F_A=-\nla^*_A\nla_AF_A-Ric\sharp F_A-F_A\sharp F_A-D_A\alpha.
\end{equation*}
Therefore,
\begin{equation*}
\begin{split}
  (\tg-\ptt)|F_A|^2
&=2|\nla_A F_A|^2+2\lag Ric\sharp F_A+F_A\sharp F_A+D_A\alpha, F_A \rag\\
&\geq 2|\nla_A F_A|^2-C(1+|Ric|+|F_A|)|F_A|^2-C|\nla_A F_A||F_A|,
\end{split}
\end{equation*}
where $C$ is a constant depending on the geometry of $X$.
\end{proof}
\subsection{Monotonicity formula}

 Let $(X,g)$ be an $n$-dimensional compact Hermitian manifold  with fundamental $(1,1)$-form $\og$ satisfying $\pt\bp\og^{n-1}=\pt\bp\og^{n-2}=0.$ We regard $X$ as a $2n$-dimensional Riemannian manifold. For any $x_0\in X$, there exist normal geodesic coordinates $\{x_i\}^{2n}_{i=1}$ in the geodesic ball $B_{r}(x_0)$ centered at $x_0$ with radius $r\leq i_{X}$ such that $x_0=(0,...,0)$ and
$$\mid g_{ij}(x)-\delta_{ij}\mid\leq C(x_0)\mid x \mid^2, \ \ \ \left|\frac{\pt g_{ij}}{\pt x_k}\right| \leq C(x_0)\mid x\mid\ \ \forall x\in B_r,$$
where $i(X)$ is the  infimum of the injectivity radius over $X$ and $C(x_0)$ a positive constant depending on $x_0$.

Let $u=(x,t)$ be a point in $X\times \mbr$. For a fixed point $u_0=(x_0, t_0)\in X\times \mbr^+$, we write
\begin{equation*}
\begin{split}
&S_r=X\times \{t=t_0-r^2\},\\
&T_r=\{u=(x,t):t_0-4r^2\leq t\leq t_0-r^2, x\in X\},\\
&P_r(u_0)=B_r(x_0)\times [t_0-r^2,t_0+r^2].
\end{split}
\end{equation*}
For simplicity, we denote $ S_r(0,0),$ $T_r(0,0)$ and $P_r(0,0)$ by $S_r,$ $T_r$ and $P_r$ respectively.

The fundamental solution of the (backward) heat equation with singularity at $(x_0,t_0)$ is
$$ G_{(x_0,t_0)}(x,t)=\frac{1}{(4\pi(t_0-t))^{2n}}\exp{\left(-\frac{|x-x_0|^2}{4(t_0-t)}\right)},\ \  t\leq t_0.$$
For simplicity, we denote $G_{(0,0)}(x,t)$ by $G(x,t).$

Assume that $A(t)$ is a smooth global solution of the heat flow (\ref{A.0}) in $X\times \mbr^+.$ Let $f$ be a smooth cut-off function such that $|f|\leq 1,$ $f\equiv 1$ on $B_{R/2}$, $f=0$ outside $B_{R}$ and $|\nabla f|\leq 2/R,$ where $R\leq i_X.$ For any $(x,t)\in X\times [0,+\infty),$ we set
$$\Phi(r)=r^2\int_{T_r(u_0)} e(A)f^2G_{u_0}\ dt$$
Then we have
\begin{thm}
Assume that $A(t)$ is a solution of the heat flow (\ref{A.0}) in $X\times R_+$ with initial data $A_0.$ Let $f$ be a smooth cut-off function such that $|f|\leq 1,$ $f\equiv 1$ on $B_{R/2}$, $f=0$ outside $B_{R}$ and $|\nabla f|\leq 2/R,$ where $R\leq i(X).$ Then for any $r_1$ and $r_2$ with $0<r_1\leq r_2\leq \min\{R/2, \sqrt{t_0}/2\},$ we have
\begin{equation*}
\Phi(r_1)\leq C\exp(C(r_2-r_1))\Phi(r_2)+C(r_2^2-r_1^2)YM(0)+CR^{2-2n}\int_{P_R(u_0)}|F_A|^2\ dV_gdt.
\end{equation*}
\end{thm}
\begin{proof}
Choose normal geodesic coordinates $\{x_i\}^{2n}_{i=1}$ in the geodesic ball $B_R(x_0)$. Setting $x=r\widetilde{x },$ $t=t_0+r^2\widetilde{t}$, we have
\begin{equation*}
\begin{split}
\Phi(r)=&r^2\int_{T_r(u_0)} e(A)f^2 G_{u_0}\ dV_g\ dt\\
=&r^2\int_{t_0-4r^2}^{t_0-r^2}\int_{\mathbb{R}^{2n}} e(A)(x,t) f^2(x)G_{u_0}(x,t) \det{(g_{ij})}\ dx\ dt\\
=&r^4\int_{\tilde{T}_1} e(A)(r\widetilde{x},t_0+r^2\tilde{t}) f^2(r\tilde{x})G(\tilde{x},\tilde{t}) \det{(g_{ij})}(r\tilde{x})\ d\tilde{x} d\tilde{t}
\end{split}
\end{equation*}
where $\tilde{T}_1=[-4,-1]\times \mathbb{R}^{2n}.$ The $r$-direction derivative of $\Phi(r)$ is
\begin{equation*}
\begin{split}
d \Phi(r)/dr=&4r^3\int_{\tilde{T}_1} e(A)(r\widetilde{x},t_0+r^2\tilde{t}) f^2(r\tilde{x})G(\tilde{x},\tilde{t}) \det{g_{ij}}(r\tilde{x})\ d\tilde{x} d\tilde{t}\\
&+r^4\int_{\tilde{T}_1}\left(\frac{d}{dr} e(A)(r\widetilde{x},t_0+r^2\tilde{t})\right)f^2(r\tilde{x})G(\tilde{x},\tilde{t}) \det{g_{ij}}(r\tilde{x})\ d\tilde{x} d\tilde{t}\\
&+\int_{\tilde{T}_1}e(A)(r\widetilde{x},t_0+r^2\tilde{t})\left(\frac{d}{dr}f^2(r\tilde{x}) \det{(g_{ij})}(r\tilde{x})\right)G(\tilde{x},\tilde{t})d\tilde{x} d\tilde{t}\\
=:&I_1+I_2+I_3,
\end{split}
\end{equation*}
where $I_1=\frac{4\Phi(r)}{r}.$
At first, we calculate $I_2.$  By simple calculation, we have
\begin{equation}\label{eq8}
\begin{split}
  \frac{\pt}{\pt r}e(A)(r\tilde{x}^i,t_0+r^2\tilde{t})
&=\frac{\pt }{\pt x^i}e(A)(r\tilde{x}^i,t_0+r^2\tilde{t})\frac{\pt x^i}{\pt r}+\frac{\pt}{\pt t} e(A)(r\tilde{x}^i,t_0+r^2\tilde{t})\frac{\pt t}{\pt r}\\
&=\tilde{x}^i\frac{\pt}{\pt x^i} e(A)+2r\tilde{t}\ptt e(A)\\
&=\frac{x^i}{r}\frac{\pt}{\pt x^i} e(A)+\frac{2(t-t_0)}{r}\ptt e(A).
\end{split}
\end{equation}
Substituting (\ref{eq8}) into $I_2$, we have
\begin{equation*}
\begin{split}
I_2=&r\int_{T_r(u_0)}\left(x^i\frac{d}{dx^i} e(A)(x,t)\right) f^2(x)G_{u_0}(x,t) \ dV_g dt\\
&+r\int_{T_r(u_0)}\left(2(t-t_0)\frac{d}{dt} e(A)(x,t)\right) f^2(x)G_{u_0}(x,t) \ dV_g dt\\
=:&I_{2.1}+I_{2.2}
\end{split}
\end{equation*}
From the Bianchi identity
$$D_A F_A=0,$$
we have
\begin{equation}\label{bianchi}
\begin{split}
0=&  D_A F_A(\frac{\pt}{\pt x^i}, \frac{\pt}{\pt x^j},\frac{\pt}{\pt x^k})\\
=&\nabla_{A,\pt/\pt x^i}F_A(\frac{\pt}{\pt x^j},\frac{\pt}{\pt x^k})-\nabla_{A,\pt/\pt x^j}F_A(\frac{\pt}{\pt x^i},\frac{\pt}{\pt x^k})+\nabla_{A,\pt/\pt x^k}F_A(\frac{\pt}{\pt x^i},\frac{\pt}{\pt x^j}).
\end{split}
\end{equation}
For simplicity, we set $\nla_{A,i}:=\nla_{A,\frac{\pt}{\pt x^i}}$ and $\nla_i:=\nla_{\frac{\pt}{\pt x^i}}$. Therefore, by (\ref{bianchi}), we have
\begin{equation}\label{eq9}
\begin{split}
&x^i\frac{\pt}{\pt x^i}|F_A|^2\\
=&2x^i\mbox{Re}\lag\nabla_{A,i}F_A, F_A\rag\\
=&x^i\mbox{Re}\Big\lag \nabla_{A,i}F_A\left(\frac{\pt}{\pt x^j},\frac{\pt}{\pt x^k}\right)dx^j\wedge dx^k, F_A\Big\rag\\
=&x^i\mbox{Re}\Big\lag \left(\nabla_{A,j}F_A(\frac{\pt}{\pt x^i},\frac{\pt}{\pt x^k})-\nabla_{A,k}F_A(\frac{\pt}{\pt x^i},\frac{\pt}{\pt x^j})\right)dx^j\wedge dx^k, F_A\Big\rag\\
=&2x^i\mbox{Re}\Big\lag \nabla_{A,j}F_A\left(\frac{\pt}{\pt x^i},\frac{\pt}{\pt x^k}\right)dx^j\wedge dx^k,F_A\Big\rag\\
=&2\mbox{Re}\Big\lag (\nabla_{A,j}x^iF_A)\left(\frac{\pt}{\pt x^i},\frac{\pt}{\pt x^k}\right)dx^j\wedge dx^k,F_A\Big\rag
-2\mbox{Re}\Big\lag\delta^j_i F_A(\frac{\pt}{\pt x^i},\frac{\pt}{\pt x^k})dx^j\wedge dx^k, F_A\Big\rag\\
=&2\mbox{Re}\Big\lag dx^j\wedge (\nabla_{A,j}x^iF_A)(\frac{\pt}{\pt x^i},\frac{\pt}{\pt x^k})dx^k, F_A \Big\rag-4|F_A|^2
\end{split}
\end{equation}
and
\begin{equation}\label{eq10}
\begin{split}
&dx^j\wedge(\nabla_{A,j}(x^iF_A))(\frac{\pt}{\pt x^i},\frac{\pt}{\pt x^k}) dx^k\\
=&dx^j\wedge\left(\nabla_{A,j}(x^iF_{A,ik})-x^iF_A(\nabla_{j}\frac{\pt}{\pt x^i},\frac{\pt}{\pt x^k})-x^iF_A(\frac{\pt}{\pt x^i} ,\nabla_{j}\frac{\pt}{\pt x^k})\right) dx^k\\
=&\nabla_{A,j}(x^iF_{A,ik})dx^j\wedge dx^k-x^iF_A(\nabla_{j}\frac{\pt}{\pt x^i},\frac{\pt}{\pt x^k})dx^j\wedge dx^k\\
=&dx^j\wedge\nabla_{A,j}(x^iF_{A,ik}dx^k)-x^iF_{A,ik}dx^j\wedge \nabla_{j}dx^k-x^iF_A(\nabla_{j}\frac{\pt}{\pt x^i},\frac{\pt}{\pt x^k})dx^j\wedge dx^k\\
=&D_A(x^i F_{A,ik}dx^k)-x^iF_A(\nabla_{j}\frac{\pt}{\pt x^i},\frac{\pt}{\pt x^k})dx^j\wedge dx^k.
\end{split}
\end{equation}
The reason for the second equality in (\ref{eq10}) is that
\begin{equation*}
\begin{split}
&\sum_{j,k}F_A(\frac{\pt}{\pt x^i} ,\nabla_{j}\frac{\pt}{\pt x^k})dx^j\wedge dx^k\\
=&(\sum_{j<k}+\sum_{k< j})F_A(\frac{\pt}{\pt x^i} ,\nabla_{j}\frac{\pt}{\pt x^k})dx^j\wedge dx^k\\
=&\sum_{j< k}(F_A(\frac{\pt}{\pt x^i} ,\nabla_{j}\frac{\pt}{\pt x^k})dx^j\wedge dx^k+F_A(\frac{\pt}{\pt x^i} ,\nabla_{k}\frac{\pt}{\pt x^j})dx^k\wedge dx^j)\\
=&0.
\end{split}
\end{equation*}
And the reason for the forth equality in (\ref{eq10}) is that $dx^j\wedge \nabla_{j}dx^k=Ddx^k=0.$ Substituting (\ref{eq10}) into (\ref{eq9}), we have
\begin{equation*}
\begin{split}
x^i\frac{\pt}{\pt x^i}|F_A|^2=2\mbox{Re}\lag D_A(x^i F_{A,ik}dx^k)-x^iF_A(\nabla_{j}\frac{\pt}{\pt x^i},\frac{\pt}{\pt x^k})dx^j\wedge dx^k, F_A\rag-4|F_A|^2.
\end{split}
\end{equation*}
Noting that $\displaystyle{\frac{\pt G_{u_0}}{\pt x^{\alpha}}=\frac{x^{\alpha}G_{u_0}}{2(t-t_0)}}$, we have
\begin{equation}\label{eq11}
\begin{split}
&\int_{T_r(u_0)}\lag d(f^2 G_{u_0})\wedge( x^i F_{A,ik}dx^k), F_A\rag \ dV_g dt\\
=&\int_{T_r(u_0)}\lag x^i F_{A,ik}dx^k, \nabla (f^2 G_{u_0})\llcorner F_A\rag \ dV_g dt\\
=&\int_{T_r(u_0)}\lag x^i F_{A,ik}dx^k, 2 f g^{\alpha\beta}\frac{\pt f}{\pt x^{\alpha}}F_{A,\beta l}dx^l \rag G_{u_0} \ dV_g dt\\
&+\int_{T_r(u_0)}\lag x^i F_{A,ik}dx^k,\frac{g^{\alpha\beta}x^{\alpha}}{2(t-t_0)}F_{A,\beta l}dx^l \rag f^2G_{u_0} \ dV_g dt
\end{split}
\end{equation}
and
\begin{equation*}
\begin{split}
&\int_{T_r(u_0)}(t-t_0)\Big\lag F_A , d(f^2 G_{u_0})\wedge \frac{\pt A}{\pt t}\Big\rag\ dV_g dt\\
=&\int_{T_r(u_0)}(t-t_0)\Big\lag 2 g^{\alpha\beta}\frac{\pt f}{\pt x^{\alpha}}F_{A,\beta l}dx^l, \frac{\pt A}{\pt t} \Big\rag fG_{u_0}\ dV_g dt\\
&+\int_{T_r(u_0)}(t-t_0)\Big\lag \frac{g^{\alpha\beta}x^{\alpha}}{2(t-t_0)}F_{A,\beta l}dx^l, \frac{\pt A}{\pt t} \Big\rag f^2G_{u_0}\ dV_g dt.
\end{split}
\end{equation*}
From the above all, we obtain
\begin{equation*}
\begin{split}
I_{2.1}=&r2\mbox{Re}\int_{T_r(u_0)}\lag D_A(x^i F_{A,ik}dx^k), F_A\rag f^2 G_{u_0}\ dV_g dt\\
&-r2\mbox{Re}\int_{T_r(u_0)}\Big\lag x^iF_A(\nabla_{j}\frac{\pt}{\pt x^i},\frac{\pt}{\pt x^k})dx^j\wedge dx^k, F_A \Big\rag f^2 G_{u_0}\ dV_g dt\\
&-4r\int_{T_r(u_0)}|F_A|^2 f^2 G_{u_0}\ dV_g dt\\
=&-r2\mbox{Re}\int_{T_r(u_0)}\Big\lag ( x^i F_{A,ik}dx^k), \frac{\pt A}{\pt t}\Big\rag f^2 G_{u_0} \ dV_g dt\\
&-r2\mbox{Re}\int_{T_r(u_0)}\lag  x^i F_{A,ik}dx^k, [\Lambda_{\og}, d\og]^*F_A \rag f^2 G_{u_0}\ dV_g dt\\
&-r2\mbox{Re}\int_{T_r(u_0)}\lag x^i F_{A,ik}dx^k, 2 f g^{\alpha\beta}\frac{\pt f}{\pt x^{\alpha}}F_{A,\beta l}dx^l \rag fG_{u_0} \ dV_g dt\\
&-r2\mbox{Re}\int_{T_r(u_0)}\lag x^i F_{A,ik}dx^k,\frac{g^{\alpha\beta}x^{\alpha}}{2(t-t_0)}F_{A,\beta l}dx^l \rag f^2G_{u_0} \ dV_g dt\\
&-r2\mbox{Re}\int_{T_r(u_0)}\lag x^iF_A\big(\nabla_{j}\frac{\pt}{\pt x^i},\frac{\pt}{\pt x^k}\big)dx^j\wedge dx^k, F_A \rag f^2 G_{u_0}\ dV_g dt\\
&-4r\int_{T_r(u_0)}|F_A|^2 f^2 G_{u_0}\ dV_g dt,
\end{split}
\end{equation*}
and
\begin{equation*}
\begin{split}
I_{2.2}=&r2\int_{T_r(u_0)}(t-t_0)\frac{\pt}{\pt t} |F_A|^2f^2 G_{u_0}\ dV_g dt\\
=&r4\mbox{Re}\int_{T_r(u_0)}(t-t_0)\Big\lag F_A,D_A\frac{\pt A}{\pt t}\Big\rag f^2 G_{u_0}\ dV_g dt\\
=&r4\mbox{Re}\int_{T_r(u_0)}(t-t_0)\Big\lag D_A^*F_A,\frac{\pt A}{\pt t}\Big\rag f^2 G_{u_0}\ dV_g dt\\
&-4r\mbox{Re}\int_{T_r(u_0)}(t-t_0)\Big\lag F_A , d(f^2 G_{u_0})\wedge \frac{\pt A}{\pt t}\Big\rag\ dV_g dt\\
=&r4\int_{T_r(u_0)}(t-t_0)\Big|\frac{\pt A}{\pt t}\Big|^2f^2 G_{u_0}\ dV_g dt\\
&-r4\mbox{Re}\int_{T_r(u_0)}(t-t_0)\Big\lag [\Lambda_{\og}, d\og]^*F_A,\frac{\pt A}{\pt t} \Big\rag f^2 G_{u_0}\ dV_g dt\\
&-r4\mbox{Re}\int_{T_r(u_0)}(t-t_0)\Big\lag 2 g^{\alpha\beta}\frac{\pt f}{\pt x^{\alpha}}F_{A,\beta l}dx^l, \frac{\pt A}{\pt t} \Big\rag fG_{u_0}\ dV_g dt\\
&-r4\mbox{Re}\int_{T_r(u_0)}(t-t_0)\Big\lag \frac{g^{\alpha\beta}x^{\alpha}}{2(t-t_0)}F_{A,\beta l}dx^l, \frac{\pt A}{\pt t} \Big\rag f^2G_{u_0}\ dV_g dt.
\end{split}
\end{equation*}

For simplicity, we set $x\odot F_A=\frac{1}{2}x^iF_{A,ik}dx^k,$ $x\cdot F_A=\frac{1}{2}x^{\alpha}g^{\alpha\beta}F_{A,\beta l} dx^l$ and $ \nabla f \cdot F_A=2g^{\alpha\beta}f^{-1}f_{\alpha}F_{A,\beta l}dx^l,$ where $f_{\alpha}=\frac{\pt f}{\pt x^{\alpha}}.$ Substituting $I_{2.1}$ and $I_{2.2}$ into $I_2$, we have

\begin{equation*}
\begin{split}
I_2=&4r\int_{T_r(u_0)}\frac{1}{|t-t_0|}\Big||t-t_0|\frac{\pt A}{\pt t}-x\odot F_A\Big|^2 f^2 G_{u_0}\ dV_g dt\\
&+4r \int_{T_r(u_0)}\frac{1}{|t-t_0|}\Big\lag x\cdot F_A-x\odot F_A, x\odot F_A-|t-t_0|\frac{\pt A}{\pt t}\Big\rag f^2 G_{u_0}\ dV_g dt\\
&+4r\int_{T_r(u_0)}\Big\lag \nabla f\cdot F_A,  |t-t_0|\frac{\pt A}{\pt t}-x\odot F_A\Big\rag f^2 G_{u_0}\ dV_g dt\\
&+4r \int_{T_r(u_0)}\left\lag [\Lambda_{\og}, d\og]^*F_A,|t-t_0|\frac{\pt A}{\pt t}-x\cdot F_A\right \rag f^2 G_{u_0}\ dV_g dt\\
&-r2\int_{T_r(u_0)}\left\lag x^iF_A(\nabla_{j}\frac{\pt}{\pt x^i},\frac{\pt}{\pt x^k})dx^j\wedge dx^k, F_A\right \rag f^2 G_{u_0}\ dV_g dt\\
&-4r\int_{T_r(u_0)}|F_A|^2 f^2 G_{u_0}\ dV_g dt.
\end{split}
\end{equation*}
By Cauchy inequality, it holds that

\begin{equation}\label{eq12}
\begin{split}
I_2\geq & r\int_{T_r(u_0)}\frac{1}{|t-t_0|}\Big||t-t_0|\frac{\pt A}{\pt t}-x\odot F_A\Big|^2 f^2 G_{u_0}\ dV_g dt\\
&-4r\int_{T_r(u_0)}\frac{1}{|t-t_0|}\big|x\cdot F_A-x\odot F_A\big|^2f^2G_{u_0}\ dV_gdt\\
&-4r\int_{T_r(u_0)}|t-t_0||\nabla f\cdot F_A|^2f^2G_{u_0}\ dV_g dt\\
&-4r\int_{T_r(u_0)}|t-t_0||[\Lambda_{\og}, d\og]^*F_A|^2f^2G_{u_0}\ dV_g dt\\
&-r2\int_{T_r(u_0)}\left\lag x^iF_A(\nabla_{j}\frac{\pt}{\pt x^i},\frac{\pt}{\pt x^k})dx^j\wedge dx^k, F_A \right\rag f^2 G_{u_0}\ dV_g dt\\
&-4r\int_{T_r(u_0)}|F_A|^2 f^2 G_{u_0}\ dV_g dt.
\end{split}
\end{equation}

By simple calculation, we have
\begin{equation}\label{eq13}
\begin{split}
I_3=&r\int_{T_r(u_0)}|F_A|^2 x^i\frac{\pt \left(f^2\sqrt{\det (g_{ij})}\right)}{\pt x^i}G_{u_0}\ dxdt\\
=&r\int_{T_r(u_0)}|F_A|^2 x^i 2ff_i G_{u_0}\ dV_gdt\\
&+r\int_{T_r(u_0)}|F_A|^2 x^i\frac{\pt \sqrt{\det (g_{ij})}}{\pt x^i}f^2G_{u_0}\ dxdt\\
=&r\int_{T_r(u_0)}|F_A|^2 x^i 2ff_i G_{u_0}\ dV_gdt\\
&+\frac{r}{2}\int_{T_r(u_0)}|F_A|^2 x^i \mbox{tr}\left(\frac{\pt g}{\pt x^i}g^{-1} \right)f^2 G_{u_0}\ dV_gdt.
\end{split}
\end{equation}
Since $|g_{ij}-\delta_{ij}|\leq C|x|^2$, $\Big|\displaystyle{\frac{\pt g_{ij}}{\pt x^k}}\Big|\leq C|x|$ and $|\Gamma^i_{jk}|\leq C|x|$, there exists a constant $C_1$ such that
\begin{equation}\label{eq14}
\begin{split}
&|x\cdot F_A-x\odot F_A|^2=\Big|\frac{1}{2}x^i(\delta^i_j-g^{ij})F_{A,jl}dx^l\Big|^2\leq C_1|x|^6|F_A|^2,\\
&\Big\lag x^iF_A(\nabla_{j}\frac{\pt}{\pt x^i},\frac{\pt}{\pt x^k})dx^j\wedge dx^k, F_A \Big\rag\leq C_1|x|^2|F_A|^2,\\
&\tr\left(\frac{\pt g}{\pt x^i }g^{-1}\right)\leq C_1|x|.
\end{split}
\end{equation}
In addition, it is easy to check that there exists a constant $C_2>0$ such that
\begin{equation}\label{eq15}
  |[\lw, d\og]^*F_A|^2\leq C_2|F_A|^2.
\end{equation}
Substituting  (\ref{eq14}) and (\ref{eq15}) into (\ref{eq13}) and (\ref{eq12}), we have there exists a constant $C_3>0$
\begin{equation*}
  \begin{split}
 \frac{d \Phi(r)}{d r}&\geq  r\int_{T_r(u_0)}\frac{1}{|t-t_0|}\Big||t-t_0|\frac{\pt A}{\pt t}-x\odot F_A\Big|^2 f^2 G_{u_0}\ dV_g dt\\
&-C_3r\int_{T_r(u_0)}\frac{|x|^6}{|t-t_0|}|F_A|^2 f^2G_{u_0}\ dV_gdt\\
&-4r\int_{T_r(u_0)}|t-t_0||\nabla f\cdot F_A|^2f^2G_{u_0}\ dV_g dt\\
&-C_3r\int_{T_r(u_0)}|t-t_0||F_A|^2f^2G_{u_0}\ dV_gdt\\
&-C_3r\int_{T_r(u_0)}|x|^2|F_A|^2f^2G_{u_0}\ dV_gdt\\
&-2r\int_{T_r(u_0)}|x||\nla f||f||F_A|^2G_{u_0}\ dV_gdt
\end{split}
\end{equation*}
From \cite[p.~99]{chstr}, we have there exists a constant $C_4>0$ such that
\begin{equation*}
\begin{split}
&r^{-1}|t-t_0||x|^6 G_{u_0}\leq C_4(1+G_{u_0}), \\
&r^{-1}|x|^2 G_{u_0}\leq C_4(1+G_{u_0})
\end{split}
\end{equation*}
holds on $T_r(u_0)$. Therefore, we get
\begin{equation*}
\begin{split}
&-C_3r\int_{T_r(u_0)}\left(\frac{|x|^6}{|t-t_0|}+|t-t_0|+|x|^2\right)|F_A|^2 f^2G_{u_0}\ dV_gdt\\
&\geq-C_5\Phi(r)-C_5r\mbox{YM}(0),
\end{split}
\end{equation*}
where $C_5$ is a positive constant depending on $C_3$ and $C_4.$
At last, we estimate the remaining two terms
$$-4r\int_{T_r(u_0)}|t-t_0||\nabla f\cdot F_A|^2f^2G_{u_0}\ dV_g dt $$ and $$-2r\int_{T_r(u_0)}|x||\nla f||f||F_A|^2G_{u_0}\ dV_gdt.$$
Since $\mbox{supp}(\nabla f)\subseteq  B_R(x_0)\setminus B_{R/2}(x_0)$, $|\nabla f|\leq 2/R$ and $|t-t_0|\leq 4r^2$, we have
\begin{equation*}
|t-t_0||\nabla f\cdot F_A|^2f^2G_{u_0}\leq \frac{16}{(4\pi)^nR^2}r^{2-2n}\exp{\left(-\frac{R^2}{64r^2}\right)}|F_A|^2
\end{equation*}
and
\begin{equation*}
|x||\nabla f||f|G_{u_0}\leq R\cdot \frac{2}{R}(4\pi)^{-n}r^{-2n}\exp\left(-\frac{R^2}{64r^2}\right).
\end{equation*}
Setting $h(r)=r^{2-2n}\exp{(-\frac{R^2}{64r^2})}$, we have
\begin{equation*}
h'(r)=\left( R^2/32-(2n-2)r^2\right)r^{-(2n+1)}\exp{(-R^2/64r^2)}.
\end{equation*}
So $h(r)\leq \left(\frac{64(n-1)}{e}\right)^{n-1}R^{2-2n}$. This implies
\begin{equation*}
\begin{split}
&-4r\int_{T_r(u_0)}|t-t_0||\nabla f\cdot F_A|^2f^2G_{u_0}\ dV_g dt\\
\geq&-\frac{C(n)r}{R^{2n}} \int_{t_0-4r^2}^{t_0-r^2}\int_{B_R(x_0)}|F_A|^2\ dV_gdt\\
\geq &-\frac{2C(n)r}{R^{2n}}\int_{t_0-R^2}^{t_0+R^2}\int_{B_R(x_0)}|F_A|^2\ dV_gdt.
\end{split}
\end{equation*}
Using the same method, we have
\begin{equation*}
\begin{split}
-2r\int_{T_r(u_0)}|x||\nla f||f||F_A|^2G_{u_0}\ dV_gdt
\geq -\frac{C(n)r}{R^{2n}}\int_{t_0-R^2}^{t_0+R^2}\int_{B_R(x_0)}|F_A|^2\ dV_gdt.
\end{split}
\end{equation*}

From the above all, we obtain that there exists a constant $C_6>0$ such that
\begin{equation*}
\begin{split}
\frac{d\Phi(r)}{dr}&=I_1+I_2+I_3\\
&\geq-C_6\Phi(r)-C_6r\mbox{\rm YM}(0)-\frac{C_6r}{R^{2n}}\int_{P_R(x_0,t_0)}|F_A|^2\ dV_gdt.
\end{split}
\end{equation*}
Setting $\mbox{YM(0)}+R^{-2n}\int_{P_{R}(x_0,t_0)}|F_A|^2\ dVgdt=B$, we have
\begin{equation*}
  \frac{d}{d r}\big(e^{C_6r}\Phi(r)\big)=e^{C_6r}\left(\Phi'(r)+C_6\Phi(r)\right)\geq -C_6re^{C_6r}B.
\end{equation*}
Integrating the two sides along $[r_1,r_2]$, we have
\begin{equation*}
  \Phi(r_1)\leq \exp{(C_7(r_2-r_1))}\Phi(r_2)+C_7(r_2^2-r_1^2)\mbox{YM(0)}+C_7R^{2-2n}\int_{P_R(x_0,t_0)}|F_A|^2\ dVgdt,
\end{equation*}
where $C_7$ is a positive constant depending on the geometry of $X$.

\end{proof}

\subsection{Small energy regularity}

\begin{thm}\label{smallenergy}
  Suppose that $A(t)$ is a smooth solution of the heat flow (\ref{A.0}), then there exist positive constants $\vps_0$ and $\delta_0$ such that
 if for some $0<R<\min\{i_X/2, \sqrt{t_0}/2\}$, the inequality
$$ R^{2-2n}\int_{P_R(u_0)}e(A)<\vps_0$$
holds, then  for any $\delta\in (0,\min\{\delta_0,1/4\})$, we have
\begin{eqnarray*}
  \sup\limits_{P_{\delta R}(x_0,t_0)}|F_A|^2<16(\delta R)^{-4}.
\end{eqnarray*}
\end{thm}

\begin{proof}
 For any $\delta\in (0,1/4]$, we define the function
 $$f(r)=(2\delta R-r)^4\sup\limits_{P_r(x_0,t_0)}|F_A|^2.$$
Since $f(r)$ is continuous and $f(2\delta R)=0$, we have $f(r)$ attains its maximum at a certain point $r_0\in [0,2\delta R)$. We claim that $f(r_0)<16,$ this means that for any $r\in [0,2\delta R)$, we have
$$(2\delta R-r)^4\sup\limits_{P_r(x_0,t_0)}|F_A|^2<16.$$
In particularly, when $r\in [0,\delta R]$, it holds that $\sup\limits_{P_r(x_0,t_0)}|F_A|^2<16/(2\delta R-r)^4<16(\delta R)^{-4}$.

Assuming that the claim is not true. This means $f(r_0)\geq 16.$ Set
$$\rho_0=(2\delta R-r_0)f(r_0)^{-1/4}<\frac{1}{2}(2\delta R-r_0)=\delta R-r_0/2.$$
Rescal the Riemannian metric $\tilde{g}=\rho_0^{-2}g$ and $t=t_1+\rho^2\tilde{t}$, where $(x_1,t_1)$ satisfies $$|F_A|^2(x_1,t_1)=\sup\limits_{P_{r_0}(x_0,t_0)}|F_A|^2.$$
Setting
\begin{eqnarray*}
  \begin{split}
    e_{\rho_0}(x,\tilde{t})=|F_A|^2_{\tilde{g}}=\rho_0^4|F_A|^2_{g},
  \end{split}
 \end{eqnarray*}
we have
\begin{equation*}
  e_{\rho_0}(x_1,0)=\rho_0^4|F_A|^2(x_1,t_1)=\rho_0^4\times \frac{f(r_0)}{(2\delta R-r_0)^4}=1
\end{equation*}
and
\begin{equation}\label{eq16}
\begin{split}
  \sup\limits_{(x,\tilde{t})\in \tilde{P}_1(x_1,0)}e_{\rho_0}(A)(x,\tilde{t})=&\rho_0^4\sup\limits_{P_{\rho_0}(x_1,t_1)}e(A)\leq \rho_0^4\sup\limits_{P_{\frac{2\delta R+r_0}{2}}(x_1,t_1)}e(A)\\
  \leq &\rho_0^4\left(\frac{2\delta R+r_0}{2} \right)^{-4}f(r_0)=16,
\end{split}
\end{equation}
where $\tilde{P}_1(x_1,0)=B_{\rho_0}(x_0)\times [-1,1]$.
From the Bochner type inequality and (\ref{eq16}), on $\tilde{P}_{1}(x_1,0)$, it holds that
\begin{eqnarray}\label{eq20}
  \begin{split}
    (\frac{\partial }{\partial \tilde{t}}-\tg_{\tilde{g}})e_{\rho_0}(x,\tilde{t})
    &=\rho_0^6(\ptt-\tg_g)e(A)(x,t)\\
    &\leq C\rho_0^6(1+|Ric|_g+|R_m|_g+|F_A|_g)e(A)\\
    &\leq C' e_{\rho_0}.
    \end{split}
\end{eqnarray}
By the parabolic mean value inequality, we have
\begin{eqnarray*}
  \begin{split}
    1=e_{\rho_0}(x_1,0)\leq C''\int_{\tilde{P}(x_1,0)}|F_A|^2_{\tilde{g}}\ dV_{\tilde{g}}d\tilde{t},
  \end{split}
\end{eqnarray*}
where $C>0$ is a constant depending on the geometry of $X$ and the initial connection $A_0.$

Choose normal geodesic coordinates centred at $x_1$ and construct cut-off function $\vph\in C^{\infty}_0(B_{R/2}(x_1))$ such that $0\leq \vph\leq 1$, $\vph\equiv 1$ on $B_{R/4}(x_1)$ and $|\nla\vph|\leq 8/R$. Taking $r_1=\rho$ and $r_2=\min\{1/4, \delta_0\}R$ and applying the monotonicity formula, we have
\begin{equation}\label{eq21}
  \begin{split}
   \int_{\tilde{P_1}(x_1,0)}|F_A|^2_{\tilde{g}}\ dV_{\tilde{g}}d\tilde{t}&=\rho_0^{2-2n}\int_{P_{\rho_0}(x_1,t_1)}|F_A|_g^2\ dV_gdt\\
   \leq &C\rho_0^2\int_{P_{\rho_0}(x_1,t_1)}|F_A|_g^2G_{(x_1,t_1+2\rho_0^2)}\vph^2\ dV_gdt\\
   \leq &C\rho_0^2\int_{T_{\rho_0}(x_1,t_1+2\rho_0^2)}|F_A|_g^2G_{(x_1,t_1+2\rho_0^2)}\vph^2\ dV_gdt\\
   \leq & Ce^{C(r_2-\rho_0)}r_2^2\int_{T_{r_2}(x_1,t_1+2\rho_0^2)}|F_A|_g^2G_{(x_1,t_1+2\rho_0^2)}\vph^2\ dV_gdt\\
   &+C(r_2^2-\rho_0^2)YM(0)+C(R/2)^{2-2n}\int_{P_{R/2}(x_1,t_1)}|F_A|^2\ dV_gdt\\
   \leq &C\delta_0^{2-2n}R^{2-2n}\int_{P_R(x_0,t_0)}|F_A|^2\ dV_gdt+C\delta_0^2R^2YM(0)\\
   &+C\delta_0^{2-2n}R^{2-2n}\int_{P_R(x_0,t_0)}|F_A|^2\ dV_gdt\\
   \leq &\tilde{C}(\delta_0^{2-2n}\vps_0+\delta_0^2R^2YM(0)),
  \end{split}
\end{equation}
where the constants depend on the geometry of $X$ and the initial data $A_0.$ From (\ref{eq20}) and (\ref{eq21}), we have $ 1\leq C''\tilde{C}(\delta_0^{2-2n}\vps_0+\delta_0^2R^2YM(0)).$ Thus, choosing  $\delta_0$ and $\vps_0$ properly,  we can obtain a contradiction .
\end{proof}
\section{Proof of Theorem \ref{main}}
Using the same argument as that in the proof of Theorem 2 in Bando and Siu's paper (\cite{bando}), we have
\begin{thm}\label{extension}
  Let $X$ be an $n$-dimensional complex manifold, $g$ a Hermitian metric on $X$ with associated $(1,1)$-form $\og$. Let $(E, h)$ be a holomorphic vector bundle with a Hermitian metric $h$ over $X\setminus S$, where $S$ is a closed subset with locally finite Hausdorff measure of real co-dimension $4$.If the curvature $F_h$ is locally integrable, then
\begin{enumerate}[(1)]
  \item $E$ can be extended to the whole $X$ as a reflexive sheaf $\mee$, and for any local section $s\in \Gamma(U,\mee)$, $\log^+ h(s,s)$ belongs to $H_{loc}^1;$
\item If $\lw F$ is locally bounded, then $h$ is locally bounded and $h\in W^{2,p}_{loc}$ for any finite $p$ where $\mee$ is locally free;
\item If $(E,h)$ is Hermitian-Einstein, then $h$ smoothly extends as a Hermitian-Einstein metric over the place where $\mee$ is locally free.
\end{enumerate}
\end{thm}

{\bf Proof of Theorem \ref{main}}

 $(1)$ { $Step 1$ Construct the closed set $\Sigma$ of Hausdorff codimension at least $4$}

From Lemma \ref{lmm1}, we have for arbitrary sequence $t_{k}\raw \infty$ and $a>0$, it holds that
\begin{eqnarray*}
    \int_{t_k-a}^{t_k+a}\int_X\Big|\ptA \Big|^2\ dV_gdt\raw 0,\ \ \ t_k\raw \infty.
  \end{eqnarray*}
Then for arbitrary $\epsilon >0$, there exists $K\in \mathbb{Z}^+$, such that when $k\geq K,$ it holds that
\begin{equation*}
  \int_{t_k-a}^{t_k+a}\int_X\Big|\ptA\Big|^2\ dV_gdt<\epsilon.
\end{equation*}

Construct the set
\begin{eqnarray*}
  \Sigma=\bigcap\limits_{0<r<i_X}\Big\{x\in X,\lim\limits_{k\raw \infty}\inf r^{4-2n}\int_{B_r(x)}e(A)(\cdot,t_k)\ dV_g\geq \vps_1\Big\}
\end{eqnarray*}
where $\vps_1$ is determined below.

For $x_1\in X\setminus \Sigma$, there exists $r_1>0$, such that when $t_k$ is sufficiently large, we have
\begin{equation*}
  r_1^{4-2n}\int_{B_{r_1}(x_1)}e(A)(\cdot,t_k)\ dV_g<\vps_1.
\end{equation*}

Set $s=t_k-r_1^2,\ \tau=t_k+r_1^2.$ Applying Lemma \ref{lmm2}, for any $t\in[s,\tau]$, we have
\begin{equation}\label{eq18}
\begin{split}
  \int_{B_{r_1/2}(x_1)}e(A)(\cdot, t)\ dV_g
&\leq \int_{B_{r_1}(x_1)}e(A)(\cdot,t_k)+2\int^{t_k+r_1^2}_{t_k-r_1^2}\int_X \Big| \ptA \Big|^2\ dV_gdt\\
&+C\left(\mbox{YM}(0)\int_{t_k-r_1^2}^{t_k+r_1^2}\int_X\Big|\ptA\Big|^2\ dV_gdt\right)^{1/2}\\
&+Ci_X\left(\mbox{YM}(0)\int_{t_k-r_1^2}^{t_k+r_1^2}\int_X\Big|\ptA\Big|^2\ dV_gdt \right)^{1/2}\\
&\leq \int_{B_{r_1}(x_1)}e(A)(\cdot,t_k)+2\epsilon+C(1+i_X)(\mbox{YM}(0)\epsilon)^{1/2}.
\end{split}
\end{equation}
Consider
\begin{equation}\label{eq19}
 r_1^{2-2n} \int_{P_{r_1/2}(x_1,t_k)}e(A)(\cdot,t)\ dV_gdt=r_1^{2-2n}\int_{t_k-(r_1/2)^2}^{t_k+(r_1/2)^2}\int_{B_{r_1/2}(x_1)}e(A)(\cdot,t)\ dV_gdt
\end{equation}
Substituting (\ref{eq18}) into (\ref{eq19}), we have
\begin{equation*}
  \begin{split}
  &r_1^{2-2n}\int_{t_k-(r_1/2)^2}^{t_k+(r_1/2)^2}\int_{B_{r_1/2}(x_1)}e(A)(\cdot,t)\ dV_gdt\\
\leq &r_1^{2-2n}\int_{t_k-(r_1/2)^2}^{t_k+(r_1/2)^2}\int_{B_{r_1}(x_1)}e(A)(\cdot,t_k)\ dV_gdt\\
&+Cr_1^{4-2n}\left(\epsilon+(1+i_X)(YM(0)\epsilon)^{1/2}\right)\\
=& \frac{1}{2}r_1^{4-2n}\int_{B_{r_1}(x_1)}e(A)(\cdot,t_k)\ dV_g+Cr_1^{4-2n}\left(\epsilon+(1+i_X)(YM(0)\epsilon)^{1/2}\right)
\end{split}
\end{equation*}
Choosing $\vps_1=\vps_0/4^{n-1}$ and $\epsilon$ such that $C2^{2n-2}r_1^{4-2n}(\epsilon+(1+i_X)(YM(0)\epsilon)^{1/2})\leq \vps_0/2 $, we have
\begin{equation*}
  \Big(\frac{r_1}{2}\Big)^{2-2n} \int_{P_{r_1/2}(x_1,t_k)}e(A)(\cdot,t)\ dV_gdt\leq \vps_0.
\end{equation*}
where $\vps_0$ is the constant in Theorem \ref{smallenergy}. Applying the small energy regularity theorem, we have
\begin{equation*}
  \sup\limits_{P_{\delta_0 r_1}(x_1,t_k)}e(A)(\cdot,\cdot)\leq C(\delta_0 r_1)^{-4},
\end{equation*}
where $\delta_0$ is the constant in Theorem \ref{smallenergy}.

It is easy to check that for any $x\in B_{\delta_0 r_1}(x_1)$, we can choose small enough $r_x$ such that $B_{r_x}(x)\subseteq B_{\delta_0 r_1}(x_1)$ and
\begin{equation*}
  \begin{split}
  r_x^{4-2n}\int_{B_{r_x}(x)}e(A)(\cdot,t_k)\leq r_x^{4-2n}r_x^{2n}C(\delta_0 r_1)^{-4}=C\left(\frac{r_x}{\delta_0 r_1}\right)^4<\vps_1,
\end{split}
\end{equation*}
 that is $B_{\delta_0 r_1}(x_1)\subseteq X\setminus \Sigma$ and $\Sigma$ is closed.

In fact $\mathcal{H}^{2n-4}(\Sigma)< \infty$. Since $\Sigma$ is closed, we have for any $\delta>0,$ there exist finite geodesic balls $\{B_{r_i}(x_i)\}_{i\in\Gamma}$, where $x_i\in \Sigma,\ r_i<\delta$, such that
\begin{itemize}
  \item $\Sigma\subseteq\cup_{i\in \Gamma} B_{r_i}$,
\item  when $ i\neq j$, $B_{r_i/2}(x_i)\cap B_{r_j/2}(x_j)=\emptyset$.
\end{itemize}
Since $x_i\in \Sigma$, we have
\begin{equation*}
r_i^{4-2n}\int_{B_{r_i/2}}e(A)(\cdot,t_k)\ dV_g>2^{2n-4}\vps_1,
\end{equation*}
for $t_k$ is sufficiently large. This implies
\begin{equation*}
  r_i^{2n-4}<\vps_1^{-1}2^{4-2n}\int_{B_{r_i/2}}e(A)(\cdot,t_k)\ dV_g.
\end{equation*}

Thus we have
\begin{equation*}
\sum_{i\in \Gamma}r_i^{2n-4}\leq \vps_1^{-1}2^{4-2n}\int_{\cup B_{r_i/2}}e(A)(\cdot,t_k)<  \vps_1^{-1}2^{4-2n}\mbox{YM(0)}<+\infty.
\end{equation*}
This implies that $\mathcal{H}^{2n-4}(\Sigma)\leq  \vps_1^{-1}2^{4-2n}\mbox{YM(0)}<\infty. $

\medskip

{ $Step 2$ The convergence in $X\setminus \Sigma$}

From the above, we have that for any $x_0\in X\setminus \Sigma,$ there exists $r_0$, when $t_k$ is sufficiently large, it holds that
$$\sup\limits_{P_{r_0}(x_0,t_k)}e(A)(\cdot,\cdot)\leq C.$$
Applying the Uhlenbeck weak compactness theorem (\cite{Uhlenbeck1}), there exists subsequence $\{t_{k'}\} \subseteq \{t_k\}$ and gauge transformations $\sigma(k')$ such that
$\sigma(k')(A(t_{k'}))$ converges to the connection $A_{\infty}$ of limiting bundle $(E_{\infty}, H_{\infty})$ in weak $W^{1,2}_{loc}(X\setminus \Sigma)$ sense and $A_{\infty}$ satisfies
\begin{equation*}
  D_{A_{\infty}}\lw F_{A_{\infty}}=0.
\end{equation*}

In fact, over $X\setminus \Sigma$, it holds that  $\sigma(k')(A(t_{k¡®}))$ converges to $A_{\infty}$ in $\mcc^{\infty}_{loc}$ sense.
For any $x_0\in X\setminus \Sigma$, there exists small enough $r_0$, such that when $t_k$ is sufficiently large, we have
\begin{equation*}
  \sup\limits_{P_{r_0}(x_0,t_{k})}|F_A|^2\leq C.
\end{equation*}
 Therefore, over $P_{r_0}(x_0,t_{k})$, we have
\begin{equation*}
\begin{split}
  (\tg-\ptt)|F_A|^2
&\geq 2|\nla_A F_A|^2-C(1+|Ric|+|Rm|+|F_A|)|F_A|^2-C|F_A||\nla_A F_A|\\
&\geq |\nla_A F_A|^2-C|F_A|^2.
\end{split}
\end{equation*}
Assume that there exist $r_j,\ \ j=0,\cdots, l-1$ such that
\begin{equation*}
  \sup\limits_{P_{r_j}(x_0, t_{k})}|\nla_A^j F_A|^2\leq C,
\end{equation*}
By a similar proof as that in Lemma \ref{bochner}, we have there exists $r_l$ such that
\begin{equation*}
  \begin{split}
  (\tg-\ptt)|\nla_A^l F_A|^2
&\geq 2|\nla_A^{l+1}F_A|^2-C|\nla^l_A F_A||\nla^{l+1}_A F_A|-C|\nla_A^l F_A|^2-C|\nla_A^l F_A | \\
&\geq |\nla^{l+1}_A F_A|^2-C|\nla_A^l F_A|^2-C.
\end{split}
\end{equation*}
in $P_{r_l}(x_0,t_{k})$.
This implies for any $j=1,\ \cdots,\ l$, we have
\begin{equation*}
  \int_{P_{r_j}(x_0,t_k)}|\nla^j_A F_A|^2\leq C.
\end{equation*}
From the parabolic mean value inequality, there exists $\delta>0$ such that
\begin{equation*}
  \sup\limits_{P_{\delta r_l}(x_0,t_k)}|\nla_A^{l}F_A|^2\leq C.
\end{equation*}
Using Donaldson's diagonal technique in \cite[Theorem ~ 4.4.8]{DonKro}, there exists subsequence $\{t_{k_i}\}$ and smooth gauge transformation $\{\sigma_{k_i}\}$ such that $\sigma_{k_i}(A(t_{k_i}))$ converges to $A_{\infty}$ in
$\mcc_{loc}^{\infty}(X\setminus \Sigma)$ sense and $A_{\infty}$ satisfies equation (\ref{eq17}).

$(2)$ Since $\sqt\lw F_{A_{\infty}}$ is parallel and $(\sqt\lw F_{A_{\infty}})^{*H_{\infty}}=\sqt\lw F_{A_{\infty}}$, we can decompose $(E_{\infty}, H_{\infty})$ according to the eigenvalues of  $\sqt\lw F_{\infty}$ over $X\setminus \Sigma$
\begin{equation*}
  E_{\infty}=\oplus_{i=1}^{l}E^i_{\infty}.
\end{equation*}

Setting $H^i_{\infty}=H_{\infty}|_{E^i_{\infty}}$,$A_{\infty}^i=A_{\infty}|_{E^i}$, we have $A_{\infty}^i$ is a Hermitian-Einstein connection on $(E_{\infty}^i, H_{\infty}^i)$, i.e.
\begin{equation*}
  \sqt\lw F_{A^i_{\infty}}=\lmd_{i}\mbox{Id}_{E_{\infty}^i}.
\end{equation*}
From Lemma \ref{lmm1}, we have $\rm{YM}(t)$ is decreasing along the flow. So it holds that
\begin{equation*}
  \int_{X\setminus \Sigma}|F_{A_{\infty}}|^2_{H_{\infty}}<\infty.
\end{equation*}
In addition that $\mathcal{H}^{2n-4}(\Sigma)<\infty$ and $H_{\infty}^i$ satisfies the Hermitian-Einstein equation, we have that every $(E_{\infty}^i, \bp_{A_{\infty}^i})$ can be extended to the whole $X$ as a reflexive sheaf (also denoted by $(E_{\infty}^i, \bp_{A_{\infty}^i})$) and $H_{\infty}^i$ can be smoothly extended over the place where the sheaf $(E_{\infty}^i,\bp_{A_{\infty}^i})$ is locally free by Theorem  \ref{extension}.
\qed

\end{document}